\documentclass[12pt]{article}

\usepackage[T1]{fontenc}
\usepackage[margin=1in]{geometry}

\usepackage{amsmath,amssymb,amsthm,mathtools}
\usepackage[dvipsnames]{xcolor}
\usepackage{microtype}
\usepackage{float}
\usepackage{tikz}
\usepackage[
    colorlinks=true,
    linkcolor=MidnightBlue,
    citecolor=MidnightBlue,
    urlcolor=MidnightBlue,
    filecolor=MidnightBlue
]{hyperref}
\newtheorem{theorem}{Theorem}
\newtheorem{lemma}{Lemma}
\newtheorem{corollary}{Corollary}

\theoremstyle{remark}
\newtheorem{remark}{Remark}

\newcommand{\R}{\mathbb{R}}

\newcommand{\nullity}{\operatorname{null}}
\newcommand{\mult}{\operatorname{mult}}

\begin{document}

\title{Multiplicity bounds for nonzero upper-Laplacian eigenvalues of simplicial complexes}

\author{Vinayak Gupta\thanks{Department of Mathematics, Indian Institute of
Technology Guwahati, Guwahati, India. Email:
\texttt{vinayakgupta1729v@gmail.com}.}}
\date{}

\hypersetup{
    pdftitle={Multiplicity bounds for nonzero upper-Laplacian eigenvalues of simplicial complexes},
    pdfauthor={Vinayak Gupta}
}

\maketitle

\begin{abstract}
Let $K$ be a finite pure ridge-connected $d$-dimensional simplicial complex
with at least two facets, and let $m_K(\lambda)$ denote the multiplicity of a
positive eigenvalue $\lambda$ of its upper Laplacian
$L_{d-1}^{\mathrm{up}}(K)$.  We obtain upper bounds for $m_K(\lambda)$ in
terms of the basic combinatorial parameters $f_d(K)$ and $f_{d-1}(K)$, the
numbers of facets and ridges, $p_d(K)$, the number of pendant facets, and
$q_d(K)$, the number of quasi-pendant ridges.  The quantity
$\theta_d(K)=d f_d(K)-f_{d-1}(K)+1$ arises as the cyclomatic number of a
natural reduced facet--ridge incidence graph.  We prove
\[
m_K(\lambda)\le 2\theta_d(K)+p_d(K),
\qquad
m_K(\lambda)\le 2\theta_d(K)+q_d(K)\quad(\lambda\ne d).
\]
We also characterize the complexes for which equality is attained in these bounds.  Finally, for every finite pure $d$-dimensional simplicial complex, we prove the sharp inequality $m_K(d)\ge p_d(K)-q_d(K)$, which is a higher-dimensional analogue of Faria's inequality.  For $d=1$, our results reduce to the corresponding multiplicity bounds for graph Laplacians.
\end{abstract}

\medskip
\noindent\textbf{Keywords:} Upper Laplacian; eigenvalue multiplicity; simplicial complex; pendant facet; ridge-connected complex; Faria inequality.

\noindent\textbf{2020 Mathematics Subject Classification:} 05E45, 05C50, 15A18.

\section{Introduction}

Spectral graph theory seeks to relate the eigenvalues of graph matrices to the
structure of the underlying graph.  A basic question in this direction is how
combinatorial parameters constrain eigenvalue multiplicities.

Let $G$ be a graph, let $L(G)=D(G)-A(G)$ be its Laplacian matrix, and write
$m_G(\lambda)$ for the multiplicity of a Laplacian eigenvalue $\lambda$.  If
$G$ is connected, let $c(G)=|E(G)|-|V(G)|+1$ be its cyclomatic number, and let
$p(G)$ and $q(G)$ be the numbers of pendant and quasi-pendant vertices,
respectively.  Grone, Merris and Sunder~\cite{GMS} proved that
$m_T(\lambda)\le p(T)-1$ for every Laplacian eigenvalue $\lambda$ of a tree
$T$.  For a general connected graph,
$m_G(\lambda)\le2c(G)+p(G)$, and the sharper bound
$m_G(\lambda)\le2c(G)+q(G)$ holds when $\lambda\ne1$; see
\cite{ChenGuoWang,GuptaQuasi}.  At the exceptional eigenvalue $1$, Faria's
inequality~\cite{Faria} gives $m_G(1)\ge p(G)-q(G)$.  Related multiplicity
results may be found in
\cite{Barik,GuoFengZhang,Andrade,AkbariKianiMirzakhah,TianWangSong} and the
references therein.

We develop higher-dimensional analogues of these results.  Let $K$ be a finite
pure $d$-dimensional simplicial complex.  Its $d$-faces are called facets and
its $(d-1)$-faces are called ridges.  We write $K_i$ for the set of
$i$-dimensional faces and $f_i(K)=|K_i|$.  After orientations are chosen, the
boundary matrix $\partial_d$ is the signed ridge--facet incidence matrix.  The
upper Laplacian acting on the ridges is
$L_{d-1}^{\mathrm{up}}(K)=\partial_d\partial_d^{\mathsf T}$, and
$m_K(\lambda)$ denotes the multiplicity of its positive eigenvalue $\lambda$.
When $d=1$, this is the ordinary graph Laplacian.

The spectra of simplicial Laplacians have been studied from several
viewpoints.  Duval and Reiner~\cite{DuvalReiner} determined the Laplacian
spectra of shifted simplicial complexes and proved their integrality.  Horak
and Jost~\cite{HorakJost} developed weighted and normalized Laplace operators
for simplicial complexes, while Steenbergen, Klivans and
Mukherjee~\cite{SteenbergenKlivansMukherjee} established a Cheeger-type
relation between small positive eigenvalues and coboundary expansion.  Song,
Wu and Fan~\cite{SongWuFan} studied the multiplicity of the extremal
eigenvalue $i+2$ of the normalized $i$th upper Laplacian, and Fan, Wu and
Wang~\cite{FanWuWang} obtained bounds for the largest combinatorial
upper-Laplacian eigenvalue.

The eigenvalue zero is closely connected with topology: its multiplicity for
the upper Laplacian is $f_{d-1}(K)-\operatorname{rank}\partial_d$, while the
kernel of the full Hodge Laplacian represents simplicial homology.
Parzanchevski and Rosenthal~\cite{ParzanchevskiRosenthal} related homology and
the upper-Laplacian spectral gap to random walks on oriented faces, and Ponoi
and Montagantirud~\cite{PonoiMontagantirud} studied random walks through a
normalized Hodge Laplacian.  In contrast, our concern is the multiplicity of
an arbitrary nonzero upper-Laplacian eigenvalue.

We use the following higher-dimensional analogues of the graph parameters
above.  A ridge is \emph{free} if it belongs to exactly one facet.  A facet is
\emph{pendant} if exactly $d$ of its $d+1$ ridges are free, and $p_d(K)$
denotes the number of pendant facets.  A nonfree ridge is
\emph{quasi-pendant} if it belongs to at least one pendant facet, and $q_d(K)$
denotes the number of such ridges.  For a pure ridge-connected complex, set
\[
\theta_d(K)=d f_d(K)-f_{d-1}(K)+1.
\]
When $d=1$, these quantities reduce to $c(G)$, $p(G)$ and $q(G)$,
respectively.

Our main upper bounds are
\[
m_K(\lambda)\le2\theta_d(K)+p_d(K),
\qquad
m_K(\lambda)\le2\theta_d(K)+q_d(K)\quad(\lambda\ne d).
\]
Here $K$ is ridge-connected, has at least two facets, and $\lambda>0$.  We
sharpen the first bound by one except when
$(\theta_d(K),p_d(K))=(1,0)$, and characterize equality in both bounds.  The
extremal complexes are simplicial cycles for the first bound, and simplicial
cycles or ridge-stars for the second.  We also determine the equality
eigenvalues for simplicial cycles.  Finally, for every finite pure
$d$-dimensional simplicial complex, we prove the sharp Faria-type inequality
$m_K(d)\ge p_d(K)-q_d(K)$.  For $d=1$, our results recover the corresponding
graph bounds.

\section{Preliminaries and simplicial notation}

We first recall the basic terminology and operators used throughout the paper.

A finite abstract simplicial complex $K$ on a finite vertex set $V$ is a collection of subsets of $V$ such that every subset of a member of $K$ also belongs to $K$.
The elements of $K$ are called \emph{faces}.  A face $\sigma$ has dimension
\(
|\sigma|-1,
\)
and
\[
\dim K=\max_{\sigma\in K}\dim\sigma.
\]
An inclusion-maximal face is called a \emph{facet}.  The complex is \emph{pure of dimension $d$} if every facet has dimension $d$.  We write $K_i$ for the set of $i$-dimensional faces and
\[
f_i(K)=|K_i|.
\]

Choose an orientation of every face of $K$.  For $i\ge0$, let $C_i(K;\R)$ be the real vector space with the oriented $i$-faces as an orthonormal basis.  For $i\ge1$, the boundary operator
\[
\partial_i:C_i(K;\R)\longrightarrow C_{i-1}(K;\R)
\]
is defined by
\[
\partial_i[v_0,\ldots,v_i]
=\sum_{j=0}^i(-1)^j
[v_0,\ldots,\widehat{v_j},\ldots,v_i].
\]
We set $\partial_0=0$.  With respect to the oriented-face bases, the same symbol $\partial_i$ denotes the boundary matrix.

The \emph{upper Laplacian} and \emph{down Laplacian} acting on $i$-chains are
\[
L_i^{\mathrm{up}}(K)
=\partial_{i+1}\partial_{i+1}^{\mathsf T},
\qquad
L_i^{\mathrm{down}}(K)
=\partial_i^{\mathsf T}\partial_i,
\]
respectively.  Their sum
\[
L_i(K)
=L_i^{\mathrm{up}}(K)+L_i^{\mathrm{down}}(K)
\]
is the $i$-dimensional combinatorial Laplacian.  This paper studies
\(
L_{d-1}^{\mathrm{up}}(K)
=\partial_d\partial_d^{\mathsf T}.
\)
For each positive eigenvalue $\lambda$ of this matrix, its multiplicity is denoted by $m_K(\lambda)$.

These Laplacians are independent of the chosen orientations up to conjugation by diagonal matrices with entries in $\{1,-1\}$, and hence their spectra are orientation-independent.

A graph $G$ is a one-dimensional simplicial complex whose vertices and edges are its $0$- and $1$-faces.  Its boundary matrix $\partial_1$ is an oriented vertex--edge incidence matrix, and therefore
\[
L_0^{\mathrm{up}}(G)
=\partial_1\partial_1^{\mathsf T}
=D(G)-A(G)
=L(G),
\]
where $D(G)$ and $A(G)$ are the degree and adjacency matrices.  For a
Laplacian eigenvalue $\lambda$ of $G$, we write $m_G(\lambda)$ for its
multiplicity.  Thus the upper Laplacian of a simplicial complex generalizes
the ordinary graph Laplacian.

A $(d-1)$-face is called a \emph{ridge}.  Its upper degree is
\[
\deg_K(\tau)=|\{F\in K_d:\tau\subset F\}|.
\]
A ridge is \emph{free} if $\deg_K(\tau)=1$.

The complex $K$ is \emph{ridge-connected} if for every pair of facets $F,F'$ there are facets
\[
F=F_0,F_1,\ldots,F_t=F'
\]
such that $F_{i-1}$ and $F_i$ share a ridge for each $i$.

A $d$-facet $F$ is \emph{pendant} if exactly $d$ of its $d+1$ ridges are free; equivalently, $F$ contains exactly one nonfree ridge.  Let
\[
p_d(K)=|\{F\in K_d:F\text{ is pendant}\}|.
\]

A nonfree ridge $\tau$ is \emph{quasi-pendant} if it is contained in at least one pendant facet.  Let
\[
q_d(K)=|\{\tau\in K_{d-1}:\tau\text{ is quasi-pendant}\}|.
\]
Every pendant facet determines exactly one quasi-pendant ridge, whereas several pendant facets may determine the same ridge.  Consequently,
\[
q_d(K)\le p_d(K).
\]

For a pure ridge-connected $d$-complex $K$, define
\[
\theta_d(K)=d f_d(K)-f_{d-1}(K)+1.
\]

A pure ridge-connected $d$-complex $K$ is called a \emph{simplicial cycle} if its facets can be cyclically ordered as
\[
F_1,F_2,\ldots,F_m,\qquad m\ge3,
\]
and there are distinct ridges $R_1,\ldots,R_m$ such that, with indices taken modulo $m$,
\begin{enumerate}
\item[(i)] $R_i=F_i\cap F_{i+1}$;
\item[(ii)] $R_i$ is contained in exactly the two facets $F_i$ and $F_{i+1}$;
\item[(iii)] $R_1,\ldots,R_m$ are all the nonfree ridges of $K$.
\end{enumerate}
Thus each facet $F_i$ has exactly two nonfree ridges, namely $R_{i-1}$ and $R_i$, and its remaining $d-1$ ridges are free.

A pure ridge-connected $d$-dimensional simplicial complex $K$ with at least two facets is called a \emph{ridge-star} if there is a ridge $R$ contained in every facet and every ridge other than $R$ is free.  A ridge-star is the higher-dimensional analogue of an ordinary star tree: when $d=1$, the common ridge $R$ is the central vertex and the facets containing it are the edges of the star.

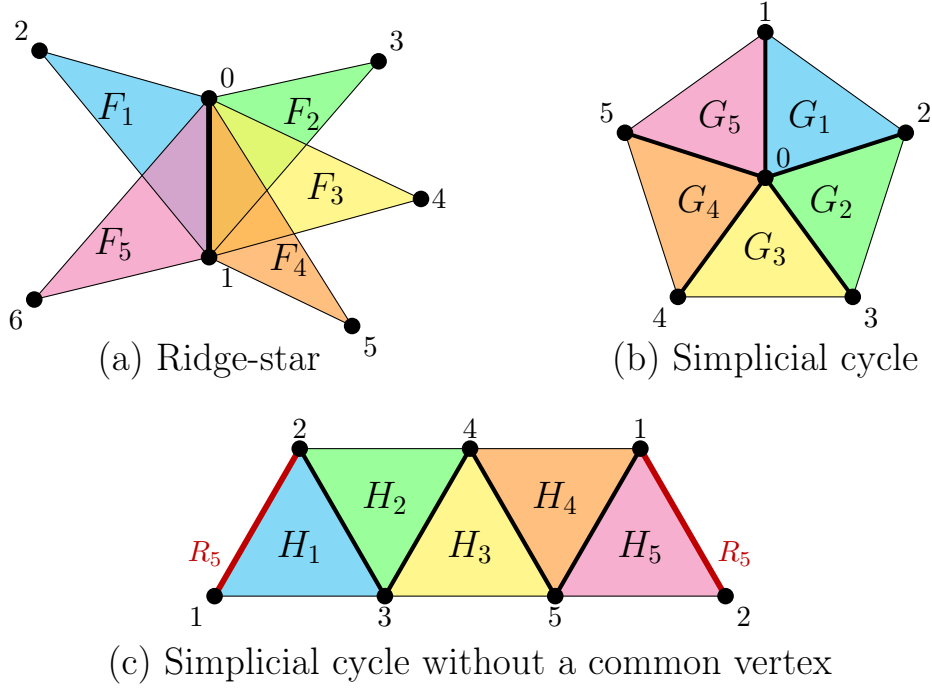
\begin{figure}[H]
\centering
\begin{tikzpicture}[scale=0.7,
    vertex/.style={circle,fill=black,inner sep=2.2pt},
    lab/.style={font=\small},
    facetlab/.style={font=\large}
]

\begin{scope}[shift={(-5.1,0)}]
\coordinate (v0) at (0,1.5);
\coordinate (v1) at (0,-1.5);
\coordinate (v2) at (-3.2,2.4);
\coordinate (v3) at (3.2,2.2);
\coordinate (v4) at (4.0,-0.4);
\coordinate (v5) at (2.7,-2.8);
\coordinate (v6) at (-3.3,-2.3);

\fill[cyan!55,opacity=.72]    (v0)--(v1)--(v2)--cycle;
\fill[green!55,opacity=.72]   (v0)--(v1)--(v3)--cycle;
\fill[yellow!75,opacity=.72]  (v0)--(v1)--(v4)--cycle;
\fill[orange!70,opacity=.72]  (v0)--(v1)--(v5)--cycle;
\fill[magenta!55,opacity=.72] (v0)--(v1)--(v6)--cycle;

\draw (v0)--(v2)--(v1);
\draw (v0)--(v3)--(v1);
\draw (v0)--(v4)--(v1);
\draw (v0)--(v5)--(v1);
\draw (v0)--(v6)--(v1);
\draw[line width=2.2pt] (v0)--(v1);

\foreach \x in {v0,v1,v2,v3,v4,v5,v6}{\node[vertex] at (\x) {};}
\node[lab,above right] at (v0) {$0$};
\node[lab,below right] at (v1) {$1$};
\node[lab,above left]  at (v2) {$2$};
\node[lab,above right] at (v3) {$3$};
\node[lab,right]       at (v4) {$4$};
\node[lab,below right] at (v5) {$5$};
\node[lab,below left]  at (v6) {$6$};

\node[facetlab] at (-1.75,1.30) {$F_1$};
\node[facetlab] at (1.75,1.20)  {$F_2$};
\node[facetlab] at (2.20,-0.20) {$F_3$};
\node[facetlab] at (1.50,-1.55) {$F_4$};
\node[facetlab] at (-1.80,-1.25) {$F_5$};
\node[font=\large] at (0,-3.45) {(a) Ridge-star};
\end{scope}

\begin{scope}[shift={(5.4,0)}]
\coordinate (u0) at (0,0);
\coordinate (u1) at (0,2.75);
\coordinate (u2) at (2.65,0.85);
\coordinate (u3) at (1.65,-2.25);
\coordinate (u4) at (-1.65,-2.25);
\coordinate (u5) at (-2.65,0.85);

\fill[cyan!55,opacity=.72]    (u0)--(u1)--(u2)--cycle;
\fill[green!55,opacity=.72]   (u0)--(u2)--(u3)--cycle;
\fill[yellow!75,opacity=.72]  (u0)--(u3)--(u4)--cycle;
\fill[orange!70,opacity=.72]  (u0)--(u4)--(u5)--cycle;
\fill[magenta!55,opacity=.72] (u0)--(u5)--(u1)--cycle;

\draw (u1)--(u2)--(u3)--(u4)--(u5)--cycle;
\draw[line width=1.5pt] (u0)--(u1) (u0)--(u2) (u0)--(u3) (u0)--(u4) (u0)--(u5);

\foreach \x in {u0,u1,u2,u3,u4,u5}{\node[vertex] at (\x) {};}
\node[lab,above right] at (u0) {$0$};
\node[lab,above]       at (u1) {$1$};
\node[lab,above right] at (u2) {$2$};
\node[lab,below right] at (u3) {$3$};
\node[lab,below left]  at (u4) {$4$};
\node[lab,above left]  at (u5) {$5$};

\node[facetlab] at (0.85,1.15) {$G_1$};
\node[facetlab] at (1.25,-0.45) {$G_2$};
\node[facetlab] at (0,-1.35) {$G_3$};
\node[facetlab] at (-1.25,-0.45) {$G_4$};
\node[facetlab] at (-0.85,1.15) {$G_5$};
\node[font=\large] at (0,-3.45) {(b) Simplicial cycle};
\end{scope}

\end{tikzpicture}
\par\medskip
\begin{tikzpicture}[scale=1.25,
    vertex/.style={circle,fill=black,inner sep=2.2pt},
    lab/.style={font=\small},
    facetlab/.style={font=\large}
]

\coordinate (w1)  at (0,0);
\coordinate (w2)  at (0.9,1.56);
\coordinate (w3)  at (1.8,0);
\coordinate (w4)  at (2.7,1.56);
\coordinate (w5)  at (3.6,0);
\coordinate (w1b) at (4.5,1.56);
\coordinate (w2b) at (5.4,0);

\fill[cyan!55,opacity=.72]    (w1)--(w2)--(w3)--cycle;
\fill[green!55,opacity=.72]   (w2)--(w3)--(w4)--cycle;
\fill[yellow!75,opacity=.72]  (w3)--(w4)--(w5)--cycle;
\fill[orange!70,opacity=.72]  (w4)--(w5)--(w1b)--cycle;
\fill[magenta!55,opacity=.72] (w5)--(w1b)--(w2b)--cycle;

\draw (w1)--(w2)--(w3)--cycle;
\draw (w2)--(w3)--(w4)--cycle;
\draw (w3)--(w4)--(w5)--cycle;
\draw (w4)--(w5)--(w1b)--cycle;
\draw (w5)--(w1b)--(w2b)--cycle;

\draw[line width=1.6pt] (w2)--(w3) (w3)--(w4)
                         (w4)--(w5) (w5)--(w1b);
\draw[red!75!black,line width=2.2pt] (w1)--(w2);
\draw[red!75!black,line width=2.2pt] (w1b)--(w2b);

\foreach \x in {w1,w2,w3,w4,w5,w1b,w2b}{\node[vertex] at (\x) {};}
\node[lab,below left] at (w1) {$1$};
\node[lab,above]      at (w2) {$2$};
\node[lab,below]      at (w3) {$3$};
\node[lab,above]      at (w4) {$4$};
\node[lab,below]      at (w5) {$5$};
\node[lab,above]      at (w1b) {$1$};
\node[lab,below right] at (w2b) {$2$};

\node[facetlab] at (0.90,0.52) {$H_1$};
\node[facetlab] at (1.80,1.04) {$H_2$};
\node[facetlab] at (2.70,0.52) {$H_3$};
\node[facetlab] at (3.60,1.04) {$H_4$};
\node[facetlab] at (4.50,0.52) {$H_5$};

\node[red!75!black,lab] at (-0.10,0.45) {$R_5$};
\node[red!75!black,lab] at (5.50,0.45) {$R_5$};
\node[font=\large] at (2.70,-0.72) {(c) Simplicial cycle without a common vertex};
\end{tikzpicture}

\caption{Three $2$-dimensional examples with five  facets: a ridge-star, a simplicial cycle with a common vertex, and an unfolded simplicial cycle with no vertex common to all facets.}
\label{fig:simplicial-examples}
\end{figure}

\begin{remark}
In the ridge-star in Figure~\ref{fig:simplicial-examples}(a), the facets are
\[
F_i=\{0,1,i+1\},\qquad 1\le i\le5.
\]
The bold vertical edge is the common ridge $R=\{0,1\}$, and every other edge is free.  In the simplicial cycle in Figure~\ref{fig:simplicial-examples}(b), the facets are
\[
G_i=\{0,i,i+1\}\quad(1\le i\le4),
\qquad
G_5=\{0,5,1\}.
\]
The five bold radial edges are precisely the nonfree ridges, while the five outer edges are free.

Figure~\ref{fig:simplicial-examples}(c) gives a simplicial cycle with facets
\[
H_1=\{1,2,3\},\quad H_2=\{2,3,4\},\quad
H_3=\{3,4,5\},\quad H_4=\{4,5,1\},\quad
H_5=\{5,1,2\}.
\]
It is drawn after cutting along the ridge $\{1,2\}$; the two red copies of that ridge are identified to close the cycle.  This complex is a triangulation of the M\"obius strip, and it has no vertex common to all facets, since $\bigcap_{i=1}^5H_i=\varnothing$.
\end{remark}

\subsection{Notation}

For a matrix $M$, the symbol $M^*$ denotes its conjugate transpose, and $\mult_M(\nu)$ denotes the algebraic multiplicity of an eigenvalue $\nu$.  The matrices $I_n$ and $J_n$ are the $n\times n$ identity matrix and all-ones matrix, respectively; their subscripts are omitted when their sizes are clear.

\section{Hermitian graph-matrix input}

We recall the matrix theorem that will be used in the proof.  If $G$ is a simple graph on $n$ vertices, let $\mathcal S(G)$ be the set of Hermitian matrices $A=(a_{ij})$ satisfying
\[
a_{ij}\ne0\quad\Longleftrightarrow\quad ij\in E(G),\qquad i\ne j,
\]
with no restriction on the diagonal entries.

\begin{theorem}[Chen--Guo--Wang \cite{ChenGuoWang}]\label{thm:CGW}
Let $G$ be a connected graph and let $A\in\mathcal S(G)$.  For every real eigenvalue $\nu$ of $A$,
\[
\mult_A(\nu)\le 2c(G)+p(G).
\]
Equality holds if and only if $G$ is a cycle and $\mult_A(\nu)=2$.  Consequently, if $G$ is not a cycle,
\[
\mult_A(\nu)\le 2c(G)+p(G)-1.
\]
\end{theorem}

The last assertion follows because multiplicity is an integer.  The cited theorem is stated for Hermitian matrices, so it applies in particular to all real symmetric matrices arising below.

We next derive a quasi-pendant refinement for a general matrix in $\mathcal S(G)$.  This is the only graph-matrix statement beyond Theorem~\ref{thm:CGW} that we need, and we prove it completely.

\begin{lemma}\label{lem:pendant-elim}
Let $G$ be a connected graph on at least three vertices, let $A\in\mathcal S(G)$, let $\nu\in\R$, and let $P$ be the set of pendant vertices of $G$.  Assume
\[
a_{vv}\ne\nu\qquad(v\in P).
\]
Set $M=A-\nu I$ and $G_0=G-P$.  Then:
\begin{enumerate}
\item[(i)] $G_0$ is connected and nonempty;
\item[(ii)] there is a Hermitian matrix $M_0$ indexed by $V(G_0)$ such that
\[
\nullity(M)=\nullity(M_0);
\]
\item[(iii)] if $|V(G_0)|\ge2$, then the off-diagonal graph of $M_0$ is exactly $G_0$;
\item[(iv)]
\[
c(G_0)=c(G),\qquad p(G_0)\le q(G).
\]
\end{enumerate}
\end{lemma}

\begin{proof}
Because $G$ has at least three vertices, two pendant vertices cannot be adjacent: if two degree-one vertices were adjacent, connectedness would force $G=K_2$.  Hence $P$ is an independent set.  Each $v\in P$ has exactly one neighbour, say $s(v)$.

Deleting the pendant vertices from a connected graph on at least three vertices leaves a nonempty connected graph.  This proves (i).

We next prove (ii).
Order the vertices with $P$ first.  Since $P$ is independent,
\[
M=
\begin{pmatrix}
D&C\\
C^*&N
\end{pmatrix},
\]
where
\[
D=\operatorname{diag}(a_{vv}-\nu:v\in P)
\]
is invertible by hypothesis.  The Schur complement
\[
M_0=N-C^*D^{-1}C
\]
satisfies
\[
\nullity(M)=\nullity(M_0),
\]
because multiplication by invertible block-triangular matrices gives
\[
\begin{pmatrix}I&0\\-C^*D^{-1}&I\end{pmatrix}
M
\begin{pmatrix}I&-D^{-1}C\\0&I\end{pmatrix}
=
\begin{pmatrix}D&0\\0&M_0\end{pmatrix}.
\]
Thus $M_0$ is Hermitian, is indexed by $V(G_0)$, and has the same nullity as $M$.  This proves (ii).

We now prove (iii).
Each row of $C$ has exactly one nonzero entry, namely in the column $s(v)$.  Therefore $C^*D^{-1}C$ is diagonal.  Thus the Schur complement changes only diagonal entries of $N$ and leaves every off-diagonal entry unchanged.  Hence, whenever $G_0$ has at least two vertices, the off-diagonal graph of $M_0$ is exactly the induced graph $G_0$.
This proves (iii).

Finally, we prove (iv).
Since each deleted pendant vertex removes exactly one vertex and exactly one edge,
\[
|E(G_0)|=|E(G)|-|P|,
\qquad
|V(G_0)|=|V(G)|-|P|.
\]
Both graphs are connected, so
\[
c(G_0)=|E(G_0)|-|V(G_0)|+1=c(G).
\]
Finally, let $u$ be pendant in $G_0$.  If $u$ had no neighbour in $P$, then its degree in $G$ would also be one, so $u$ itself would lie in $P$, a contradiction.  Thus every pendant vertex of $G_0$ is adjacent in $G$ to at least one pendant vertex of $G$, hence is quasi-pendant in $G$.  Therefore $p(G_0)\le q(G)$.
This proves (iv).
\end{proof}

\begin{theorem}\label{thm:Hermitian-q}
Let $G$ be connected on at least three vertices, let $A\in\mathcal S(G)$, and let $\nu\in\R$.  If
\[
a_{vv}\ne\nu
\]
for every pendant vertex $v$, then
\[
\mult_A(\nu)\le 2c(G)+q(G).
\]
If $G$ is neither a star nor a cycle, then
\[
\mult_A(\nu)\le 2c(G)+q(G)-1.
\]
\end{theorem}

\begin{proof}
If $\nu$ is not an eigenvalue there is nothing to prove, so assume it is.

Let $G_0=G-P$, and let $M_0$ be the Hermitian matrix supplied by Lemma~\ref{lem:pendant-elim}(ii).  If $G_0$ consists of one vertex, then connectedness of $G$ implies that $G$ is a star.  In this case $q(G)=1$ and
\[
\mult_A(\nu)=\nullity(M_0)\le1=q(G),
\]
so the first assertion holds.

Suppose now that $|V(G_0)|\ge2$.  By Lemma~\ref{lem:pendant-elim}, $M_0$ is a Hermitian matrix whose graph is the connected graph $G_0$.  Applying Theorem~\ref{thm:CGW} to the eigenvalue $0$ of $M_0$ gives
\[
\mult_A(\nu)=\nullity(M_0)\le 2c(G_0)+p(G_0)\le 2c(G)+q(G).
\]
This proves the first bound.

Assume that $G$ is neither a star nor a cycle.  If $P=\varnothing$, then $q(G)=0$ and Theorem~\ref{thm:CGW}, applied directly to $A$, gives
\[
\mult_A(\nu)\le2c(G)-1.
\]
Now let $P\ne\varnothing$.  The case $|V(G_0)|=1$ would make $G$ a star, so $|V(G_0)|\ge2$.  If $G_0$ is not a cycle, then Theorem~\ref{thm:CGW} and Lemma~\ref{lem:pendant-elim} give
\[
\mult_A(\nu)
\le2c(G_0)+p(G_0)-1
\le2c(G)+q(G)-1.
\]
If $G_0$ is a cycle, then $c(G)=c(G_0)=1$ and Theorem~\ref{thm:CGW} gives $\nullity(M_0)\le2$.  Since $P\ne\varnothing$, at least one vertex of $G_0$ is adjacent in $G$ to a deleted pendant vertex, so $q(G)\ge1$.  Hence
\[
\mult_A(\nu)\le2\le2c(G)+q(G)-1.
\]
This proves the refinement.
\end{proof}

\section{Symmetric linearization of the boundary operator}

Fix a pure ridge-connected $d$-complex $K$ with at least two facets, and let
\[
B_d=\partial_d
\]
be its boundary matrix from oriented $d$-faces to oriented ridges.  Let $\lambda>0$ and put
\[
\mu=\sqrt\lambda.
\]
Define the real symmetric matrix
\[
\mathcal A_K=
\begin{pmatrix}
0&B_d\\
B_d^{\mathsf T}&0
\end{pmatrix},
\]
where the first block of coordinates is indexed by ridges and the second by facets.

\begin{lemma}\label{lem:linearization}
For every $\lambda>0$,
\[
m_K(\lambda)=\mult_{\mathcal A_K}(\sqrt\lambda)
=\nullity(\mathcal A_K-\sqrt\lambda I).
\]
\end{lemma}

\begin{proof}
Let the nonzero singular values of $B_d$ be $\sigma_1,\ldots,\sigma_r$, repeated according to multiplicity.  The positive eigenvalues of $B_dB_d^{\mathsf T}$ are $\sigma_1^2,\ldots,\sigma_r^2$, while the nonzero eigenvalues of the symmetric block matrix $\mathcal A_K$ are
\[
\sigma_1,-\sigma_1,\ldots,\sigma_r,-\sigma_r
\]
with the same multiplicities.  Thus the multiplicity of $\lambda=\mu^2>0$ in $B_dB_d^{\mathsf T}$ equals the multiplicity of $\mu$ in $\mathcal A_K$.
\end{proof}

We now separate the free-ridge coordinates in order to eliminate them.  Let $\mathcal F$ be the set of free ridges and let
\[
\mathcal N=K_{d-1}\setminus\mathcal F
\]
be the set of nonfree ridges.  Put $\ell=|\mathcal F|$, and write the boundary matrix by row blocks as
\[
B_d=
\begin{pmatrix}
B_{\mathcal F}\\
B_{\mathcal N}
\end{pmatrix},
\]
where $B_{\mathcal F}$ contains the rows indexed by free ridges and $B_{\mathcal N}$ contains the rows indexed by nonfree ridges.

Order the coordinates of $\mathcal A_K-\mu I$ as
\[
\mathcal F,\qquad \mathcal N,\qquad K_d.
\]
With respect to these three coordinate groups, the matrix is
\[
\mathcal A_K-\mu I=
\begin{pmatrix}
-\mu I_\ell&0&B_{\mathcal F}\\
0&-\mu I_{|\mathcal N|}&B_{\mathcal N}\\
B_{\mathcal F}^{\mathsf T}&B_{\mathcal N}^{\mathsf T}&-\mu I_{f_d(K)}
\end{pmatrix}.
\]
To take the Schur complement of the free-ridge block, combine the last two coordinate groups, $\mathcal N$ and $K_d$, into one block.  Then
\[
\mathcal A_K-\mu I=
\begin{pmatrix}
-\mu I_\ell&C\\
C^{\mathsf T}&D
\end{pmatrix},
\]
where
\[
C=\begin{pmatrix}0&B_{\mathcal F}\end{pmatrix},
\qquad
D=
\begin{pmatrix}
-\mu I_{|\mathcal N|}&B_{\mathcal N}\\
B_{\mathcal N}^{\mathsf T}&-\mu I_{f_d(K)}
\end{pmatrix}.
\]

\begin{lemma}\label{lem:free-elim}
Let
\[
S_\mu=D+\frac1\mu C^{\mathsf T}C.
\]
Then
\[
\nullity(\mathcal A_K-\mu I)=\nullity(S_\mu).
\]
Moreover, $C^{\mathsf T}C$ is diagonal.  Consequently, the off-diagonal graph of $S_\mu$ has vertices indexed by all facets and all nonfree ridges, with a facet $F$ adjacent to a nonfree ridge $\tau$ exactly when $\tau\subset F$.
\end{lemma}

\begin{proof}
Since $\mu>0$, the block $-\mu I_\ell$ is invertible.  Its Schur complement is
\[
D-C^{\mathsf T}(-\mu I_\ell)^{-1}C
=D+\frac1\mu C^{\mathsf T}C=S_\mu,
\]
Indeed, multiplication by the following invertible block-triangular matrices eliminates the off-diagonal blocks:
\[
\begin{pmatrix}
I&0\\
-C^{\mathsf T}(-\mu I_\ell)^{-1}&I
\end{pmatrix}
(\mathcal A_K-\mu I)
\begin{pmatrix}
I&-(-\mu I_\ell)^{-1}C\\
0&I
\end{pmatrix}
=
\begin{pmatrix}
-\mu I_\ell&0\\
0&S_\mu
\end{pmatrix}.
\]
The two multiplying matrices are invertible, so they do not change nullity.  Since $\mu \ne 0$, the block-diagonal matrix on the right has nullity $\nullity(S_\mu)$.  Hence
\[
\nullity(\mathcal A_K-\mu I)=\nullity(S_\mu).
\]

A free ridge belongs to exactly one facet.  Therefore every row of $C$ has exactly one nonzero entry, equal to $1$ or $-1$.  If $i\ne j$ are two columns of $C$, then no row has nonzero entries in both columns.  Hence
\[
(C^{\mathsf T}C)_{ij}
=\sum_{\tau\in\mathcal F}C_{\tau i}C_{\tau j}=0
\qquad(i\ne j).
\]
Thus every off-diagonal entry of $C^{\mathsf T}C$ is zero, so $C^{\mathsf T}C$ is diagonal.  The off-diagonal entries of $S_\mu$ are therefore exactly those already present in $D$, namely the nonzero incidence entries $\pm1$ between nonfree ridges and facets.
\end{proof}

Denote the off-diagonal graph of $S_\mu$ by $H_K$.  The graph is independent of $\lambda$ and of all orientation choices.

\begin{lemma}\label{lem:HK-diag}
The graph $H_K$ is connected.  If $r_F$ is the number of free ridges contained in a facet $F$, then
\[
(S_\mu)_{\tau\tau}=-\mu
\]
for every nonfree ridge $\tau$, and
\[
(S_\mu)_{FF}=-\mu+\frac{r_F}{\mu}
=\frac{r_F-\lambda}{\sqrt\lambda}
\]
for every facet $F$.  In particular, if $F$ is pendant then
\[
(S_\mu)_{FF}=\frac{d-\lambda}{\sqrt\lambda}.
\]
\end{lemma}

\begin{proof}
If two distinct facets $F$ and $F'$ share a ridge $\tau$, then $\tau$ is contained in both $F$ and $F'$, so $\deg_K(\tau)\ge2$ and $\tau$ is nonfree.  Lemma~\ref{lem:free-elim} therefore shows that $\tau$ is a vertex of $H_K$ adjacent to both $F$ and $F'$, giving the path $F-\tau-F'$ of length two.  Ridge-connectedness connects all facet vertices in this way.  Every nonfree ridge belongs to at least two facets, so every nonfree-ridge vertex is attached to this connected set.  Hence $H_K$ is connected.

We next compute the diagonal entries of $S_\mu=D+\mu^{-1}C^{\mathsf T}C$.  The diagonal entries of $D$ are all equal to $-\mu$.  If $\tau$ is a nonfree ridge, then the column of $C$ indexed by $\tau$ is zero, since $C=(0,B_{\mathcal F})$.  Hence $(C^{\mathsf T}C)_{\tau\tau}=0$, and therefore
\[
(S_\mu)_{\tau\tau}=-\mu.
\]
Now let $F$ be a facet.  For a free ridge $\tau$, the entry $C_{\tau F}$ is $1$ or $-1$ if $\tau\subset F$, and it is $0$ otherwise.  Thus the nonzero entries in the $F$-column of $C$ correspond exactly to the free ridges contained in $F$.  Consequently,
\[
(C^{\mathsf T}C)_{FF}
=\sum_{\tau\in\mathcal F}C_{\tau F}^2=r_F.
\]
It follows that
\[
(S_\mu)_{FF}=-\mu+\frac{r_F}{\mu}
=\frac{r_F-\lambda}{\sqrt\lambda},
\]
where $\mu=\sqrt\lambda$.  Finally, a pendant facet has exactly $d$ free ridges, so $r_F=d$, giving the last formula.
\end{proof}

\begin{lemma}\label{lem:simplicial-cycle}
The reduced incidence graph $H_K$ is a cycle if and only if $K$ is a simplicial cycle.
\end{lemma}

\begin{proof}
Suppose first that $H_K$ is a cycle.  Its vertices are of two types: facet vertices and nonfree-ridge vertices.  By Lemma~\ref{lem:free-elim}, the only nonzero off-diagonal entries of $S_\mu$ are the incidence entries between a facet and a nonfree ridge.  Hence every edge of $H_K$ joins vertices of different types, while there are no edges between two facets or between two nonfree ridges.  Therefore $H_K$ is bipartite, with the facets and nonfree ridges as its two vertex classes.  The two types of vertices must consequently alternate around the cycle, and there are equally many of each type.  Labeling them successively gives
\[
F_1,R_1,F_2,R_2,\ldots,F_m,R_m,F_1.
\]
Every vertex has degree two.  Hence $R_i$ is contained in exactly $F_i$ and $F_{i+1}$, every facet has precisely the two nonfree ridges $R_{i-1}$ and $R_i$, and the $R_i$ are all the nonfree ridges of $K$.  Moreover $m\ge3$, since two distinct simplices of the same dimension cannot share two distinct ridges.  Therefore $K$ is a simplicial cycle.

Conversely, if $K$ is a simplicial cycle, its definition shows directly that $H_K$ is the cycle
\[
F_1,R_1,F_2,R_2,\ldots,F_m,R_m,F_1.
\]
\end{proof}

\begin{lemma}\label{lem:ridge-star}
The reduced incidence graph $H_K$ is a star if and only if $K$ is a ridge-star.
\end{lemma}

\begin{proof}
Suppose that $H_K$ is a star.  Every nonfree-ridge vertex has degree equal to its upper degree, which is at least two, so no such vertex can be a leaf.  Hence the center of the star is a nonfree-ridge vertex, say $R$, and every leaf is a facet vertex.  It follows that every facet contains $R$ and contains no other nonfree ridge.  Thus all ridges other than $R$ are free, and $K$ is a ridge-star.

Conversely, if $K$ is a ridge-star with common ridge $R$, then the nonfree ridges consist only of $R$.  In $H_K$, the vertex $R$ is adjacent to every facet vertex, and there are no other edges.  Therefore $H_K$ is a star.
\end{proof}

\begin{lemma}\label{lem:parameters}
For a pure ridge-connected $d$-complex $K$ with at least two facets,
\[
c(H_K)=\theta_d(K),\qquad
p(H_K)=p_d(K),\qquad
q(H_K)=q_d(K).
\]
\end{lemma}

\begin{proof}
Let $r$ be the number of free ridges.  The vertices of $H_K$ consist of all $f_d(K)$ facets and the $f_{d-1}(K)-r$ nonfree ridges, so
\[
|V(H_K)|=f_d(K)+f_{d-1}(K)-r.
\]
There are $(d+1)f_d(K)$ total facet--ridge incidences.  Each free ridge contributes exactly one such incidence, and all of these $r$ incidences are absent from $H_K$.  Therefore
\[
|E(H_K)|=(d+1)f_d(K)-r.
\]
By Lemma~\ref{lem:HK-diag}, $H_K$ is connected.  Hence
\[
\begin{aligned}
c(H_K)
&=|E(H_K)|-|V(H_K)|+1\\
&=d f_d(K)-f_{d-1}(K)+1
=\theta_d(K).
\end{aligned}
\]

A nonfree-ridge vertex of $H_K$ has degree $\deg_K(\tau)\ge2$, and so it is never pendant.  A facet vertex $F$ has degree equal to the number of nonfree ridges of $F$.  Since $K$ is ridge-connected and has at least two facets, every facet has at least one nonfree ridge.  Therefore $F$ is pendant in $H_K$ exactly when it has exactly one nonfree ridge, equivalently exactly $d$ free ridges.  Hence
\[
p(H_K)=p_d(K).
\]

All pendant vertices of $H_K$ are thus facet vertices.  A vertex of $H_K$ is quasi-pendant precisely when it is a nonfree ridge adjacent to at least one pendant facet.  These are exactly the quasi-pendant ridges of $K$, and hence
\[
q(H_K)=q_d(K).
\]
\end{proof}

\section{Main multiplicity bounds}

\begin{theorem}\label{thm:main-p}
Let $K$ be a finite pure ridge-connected $d$-dimensional simplicial complex with at least two facets.  Every positive eigenvalue $\lambda$ of $L^{\mathrm{up}}_{d-1}(K)$ satisfies
\[
m_K(\lambda)\le2\theta_d(K)+p_d(K).
\]
If $(\theta_d(K),p_d(K))\ne(1,0)$, then
\[
m_K(\lambda)\le2\theta_d(K)+p_d(K)-1.
\]
Moreover, the following are equivalent:
\begin{enumerate}
\item[(i)] The first bound is attained by some positive eigenvalue;
\item[(ii)] $K$ is a simplicial cycle;
\item[(iii)] $(\theta_d(K),p_d(K))=(1,0)$.
\end{enumerate}
\end{theorem}

\begin{proof}
Put $\mu=\sqrt\lambda$.  Lemmas~\ref{lem:linearization} and \ref{lem:free-elim} give
\[
m_K(\lambda)=\nullity(S_\mu).
\]
By Lemmas~\ref{lem:free-elim} and \ref{lem:HK-diag}, $S_\mu$ is a real symmetric matrix whose graph is the connected graph $H_K$.  Therefore Theorem~\ref{thm:CGW} and Lemma~\ref{lem:parameters} give
\[
m_K(\lambda)\le2c(H_K)+p(H_K)
=2\theta_d(K)+p_d(K).
\]
If $(\theta_d(K),p_d(K))\ne(1,0)$, then $H_K$ cannot be a cycle.  The one-unit improvement now follows from Theorem~\ref{thm:CGW}.

By Lemma~\ref{lem:simplicial-cycle}, $H_K$ is a cycle exactly when $K$ is a simplicial cycle.  Since $H_K$ is connected, it is a cycle exactly when $c(H_K)=1$ and $p(H_K)=0$.  Lemma~\ref{lem:parameters} therefore gives
\[
H_K\text{ is a cycle}
\quad\Longleftrightarrow\quad
K\text{ is a simplicial cycle}
\quad\Longleftrightarrow\quad
(\theta_d(K),p_d(K))=(1,0).
\]
This proves (ii)$\Longleftrightarrow$(iii).

If equality holds for some positive eigenvalue, then equality holds in Theorem~\ref{thm:CGW} for $S_{\sqrt\lambda}$.  Hence $H_K$ is a cycle, proving (i)$\Rightarrow$(ii).

Conversely, suppose that $K$ is a simplicial cycle with $m\ge3$ facets.  Its facet Gram matrix has the form
\[
B_d^{\mathsf T}B_d=(d+1)I+A^s(C_m),
\]
where $A^s(C_m)$ is a signed cycle adjacency matrix.  After switching signs, its spectrum is either
\[
\left\{d+1+2\cos\frac{2\pi j}{m}:0\le j<m\right\}
\]
or
\[
\left\{d+1+2\cos\frac{(2j+1)\pi}{m}:0\le j<m\right\}.
\]
In the first case the indices $j=1$ and $j=m-1$ give the same positive eigenvalue, while in the second case the indices $j=0$ and $j=m-1$ do so.  These indices are distinct because $m\ge3$.  Thus there is a positive eigenvalue of multiplicity at least two.  The positive spectra of $B_d^{\mathsf T}B_d$ and $B_dB_d^{\mathsf T}$ agree, while Theorem~\ref{thm:main-p} and Lemma~\ref{lem:parameters} force this multiplicity to be exactly
\[
2=2\theta_d(K)+p_d(K).
\]
Hence equality holds, proving (ii)$\Rightarrow$(i).
\end{proof}

\begin{corollary}\label{cor:equality-spectrum}
Let $K$ be a simplicial cycle with cyclically ordered facets $F_1,\ldots,F_m$.  Put
\[
s_i=\left\langle\partial_dF_i,\partial_dF_{i+1}\right\rangle,
\qquad
\varepsilon(K)=\prod_{i=1}^m s_i.
\]
The equality eigenvalues, listed without repetition, are
\[
\lambda_j=d+1+2\cos\frac{2\pi j}{m},
\qquad
1\le j\le\left\lceil\frac m2\right\rceil-1,
\]
if $\varepsilon(K)=1$, and
\[
\lambda_j=d+1+2\cos\frac{(2j+1)\pi}{m},
\qquad
0\le j\le\left\lfloor\frac m2\right\rfloor-1,
\]
if $\varepsilon(K)=-1$.  Each has multiplicity two.
\end{corollary}

\begin{proof}
Consecutive facets share exactly one ridge, so $s_i\in\{1,-1\}$, while nonconsecutive facets share no ridge.  Thus
\[
B_d^{\mathsf T}B_d=(d+1)I+A^s(C_m).
\]
Reorienting a facet changes two adjacent signs, and reorienting a ridge changes neither their product nor the Gram matrix.  Hence $\varepsilon(K)$ is orientation-independent.  Switching reduces the signed cycle to the all-positive cycle when $\varepsilon(K)=1$ and to a cycle with one negative edge when $\varepsilon(K)=-1$.  The corresponding Fourier spectra are the two spectra displayed in the proof of Theorem~\ref{thm:main-p}.  Pairing the indices $j$ with $m-j$ in the first case and with $m-1-j$ in the second gives precisely the stated double eigenvalues; the remaining indices are unpaired.  Their angles lie strictly between $0$ and $\pi$, so they are positive.  Since $2\theta_d(K)+p_d(K)=2$, these and only these eigenvalues attain equality.
\end{proof}

The preceding bound can be improved further by replacing the number of pendant facets with the number of quasi-pendant ridges.  For this sharper estimate, however, the exceptional eigenvalue $\lambda=d$ must be excluded.

\begin{theorem}\label{thm:main-q}
Let $K$ be a finite pure ridge-connected $d$-dimensional simplicial complex with at least two facets.  Let $\lambda>0$ be an eigenvalue of $L^{\mathrm{up}}_{d-1}(K)$ with $\lambda\ne d$.
Then
\[
m_K(\lambda)\le2\theta_d(K)+q_d(K).
\]
If $K$ is neither a ridge-star nor a complex satisfying
\[
(\theta_d(K),p_d(K))=(1,0),
\]
then
\[
m_K(\lambda)\le2\theta_d(K)+q_d(K)-1.
\]
Moreover, the following are equivalent:
\begin{enumerate}
\item[(i)] The first bound is attained by some positive eigenvalue $\lambda\ne d$;
\item[(ii)] $K$ is either a ridge-star or a simplicial cycle.
\end{enumerate}
\end{theorem}

\begin{proof}
As above,
\[
m_K(\lambda)=\nullity(S_\mu),\qquad \mu=\sqrt\lambda,
\]
and $S_\mu$ has graph $H_K$.

By Lemma~\ref{lem:parameters}, the pendant vertices of $H_K$ are exactly the pendant facets of $K$.  If $F$ is such a facet, Lemma~\ref{lem:HK-diag} gives
\[
(S_\mu)_{FF}=\frac{d-\lambda}{\sqrt\lambda}\ne0.
\]
Thus the hypothesis of Theorem~\ref{thm:Hermitian-q} is satisfied for the matrix $S_\mu$ and the eigenvalue $0$.  Consequently
\[
\begin{aligned}
m_K(\lambda)
&\le2c(H_K)+q(H_K)\\
&=2\theta_d(K)+q_d(K),
\end{aligned}
\]
using Lemma~\ref{lem:parameters}.

The one-unit refinement in Theorem~\ref{thm:Hermitian-q} can fail only when $H_K$ is a star or a cycle.  By Lemma~\ref{lem:ridge-star}, the star case is equivalent to $K$ being a ridge-star.  By Lemma~\ref{lem:parameters}, the cycle case is equivalent to $(\theta_d(K),p_d(K))=(1,0)$.  The claimed refinement follows.

We now characterize equality.  If the first bound is attained for some $\lambda\ne d$, then the one-unit refinement shows that $H_K$ must be a star or a cycle.  Lemmas~\ref{lem:ridge-star} and \ref{lem:simplicial-cycle} show that $K$ is respectively a ridge-star or a simplicial cycle.  This proves (i)$\Rightarrow$(ii).

Conversely, suppose first that $K$ is a ridge-star with $m\ge2$ facets.  After orienting the facets so that their coefficients on the common ridge agree,
\[
B_d^{\mathsf T}B_d=dI_m+J_m.
\]
Thus $\lambda=d+m$ has multiplicity one.  Since $\theta_d(K)=0$ and $q_d(K)=1$, we have
\[
m_K(d+m)=1=2\theta_d(K)+q_d(K),
\]
and $d+m\ne d$.

Finally, let $K$ be a simplicial cycle with $m\ge3$ facets.  Then $\theta_d(K)=1$ and $q_d(K)=0$, so the first bound equals two.  Corollary~\ref{cor:equality-spectrum} supplies eigenvalues of multiplicity two.  At least one of them differs from $d$: if $\varepsilon(K)=-1$, the index $j=0$ gives
\[
\lambda_0=d+1+2\cos\frac{\pi}{m}>d;
\]
if $\varepsilon(K)=1$ and $m\ge4$, the index $j=1$ gives
\[
\lambda_1=d+1+2\cos\frac{2\pi}{m}>d.
\]
The remaining formal possibility $m=3$ and $\varepsilon(K)=1$ cannot occur.  Indeed, write $F_1=R\cup\{a\}$ and $F_2=R\cup\{b\}$, where $R=F_1\cap F_2$.  Since the other two shared ridges are distinct from $R$, the third facet must have the form
\[
F_3=(R\setminus\{c\})\cup\{a,b\}
\]
for some $c\in R$.  Thus $F_1,F_2,F_3$ are codimension-one faces of the common $(d+1)$-simplex on $R\cup\{a,b\}$.  Orienting them as boundary facets of that simplex, the identity $\partial_d\partial_{d+1}=0$ shows that the two facet boundaries induce opposite signs on each shared ridge.  Hence all three cycle signs are $-1$ and $\varepsilon(K)=-1$.  Therefore every simplicial cycle has an equality eigenvalue distinct from $d$, proving (ii)$\Rightarrow$(i).
\end{proof}

When $d=1$, a pure ridge-connected simplicial complex with at least two facets is a connected graph with at least two edges.  Its ridges are vertices, and
\[
\theta_1(G)=c(G),\qquad
p_1(G)=p(G),\qquad
q_1(G)=q(G),\qquad
L^{\mathrm{up}}_0(G)=L(G).
\]
Here $p(G)$ and $q(G)$ denote the numbers of pendant and quasi-pendant vertices, respectively.

\begin{corollary}\label{cor:graph-p}
Let $G$ be a connected graph with at least two edges.  Every positive Laplacian eigenvalue $\lambda$ satisfies
\[
m_G(\lambda)\le2c(G)+p(G).
\]
If $G$ is not a cycle, then
\[
m_G(\lambda)\le2c(G)+p(G)-1.
\]
Equality in the first bound is attained by some positive eigenvalue if and only if $G$ is a cycle.
\end{corollary}

\begin{proof}
Apply Theorem~\ref{thm:main-p} with $d=1$.  Simplicial cycles in dimension one are precisely the ordinary cycle graphs.
\end{proof}

\begin{corollary}\label{cor:graph-q}
Let $G$ be a connected graph with at least two edges, and let $\lambda>0$ be a Laplacian eigenvalue with $\lambda\ne1$.  Then
\[
m_G(\lambda)\le2c(G)+q(G).
\]
If $G$ is neither a star nor a cycle, then
\[
m_G(\lambda)\le2c(G)+q(G)-1.
\]
Equality in the first bound is attained by some positive eigenvalue $\lambda\ne1$ if and only if $G$ is a star or a cycle.
\end{corollary}

\begin{proof}
Apply Theorem~\ref{thm:main-q} with $d=1$.  Ridge-stars become ordinary star graphs, while simplicial cycles become ordinary cycle graphs.
\end{proof}

\section{A simplicial Faria inequality}

The exceptional eigenvalue $d$ satisfies a lower bound complementary to the preceding upper bound for $\lambda\ne d$.

\begin{theorem}\label{thm:simplicial-faria}
Let $\Upsilon$ be a finite pure $d$-dimensional simplicial complex, where $d\ge1$.  Then
\[
m_{\Upsilon}(d)\ge p_d(\Upsilon)-q_d(\Upsilon).
\]
\end{theorem}

\begin{proof}
Choose orientations of the facets and ridges of $\Upsilon$, and put
\[
B=\partial_d,\qquad M=B^{\mathsf T}B.
\]
Since $d>0$, the multiplicity of $d$ as an eigenvalue of $M$ equals its multiplicity as a nonzero eigenvalue of
\[
BB^{\mathsf T}=L_{d-1}^{\mathrm{up}}(\Upsilon).
\]

For each quasi-pendant ridge $R$, let
\[
\mathcal P_R
=\{F\in\Upsilon_d:F\text{ is pendant and its unique nonfree ridge is }R\},
\qquad
t_R=|\mathcal P_R|.
\]
The sets $\mathcal P_R$ partition the pendant facets, and therefore
\[
\sum_R t_R=p_d(\Upsilon),
\]
where the sum is over the $q_d(\Upsilon)$ quasi-pendant ridges.

For $F\in\mathcal P_R$, let $\varepsilon_F\in\{1,-1\}$ be the incidence sign of $R$ in $\partial_dF$, and define
\[
U_R=
\left\{
x\in\R^{\Upsilon_d}:
\operatorname{supp}(x)\subseteq\mathcal P_R,\quad
\sum_{F\in\mathcal P_R}\varepsilon_Fx_F=0
\right\}.
\]
The displayed relation is one nonzero linear equation on the $t_R$ coordinates, so
\[
\dim U_R=t_R-1.
\]

We claim that every $x\in U_R$ satisfies $Mx=dx$.  If $F\in\mathcal P_R$, then the only ridge of $F$ that can belong to another facet is $R$.  Since a $d$-simplex has $d+1$ ridges,
\[
\begin{aligned}
(Mx)_F
&=(d+1)x_F+
\sum_{\substack{G\in\mathcal P_R\\G\ne F}}
\varepsilon_F\varepsilon_Gx_G\\
&=(d+1)x_F+
\varepsilon_F\left(-\varepsilon_Fx_F\right)
=dx_F.
\end{aligned}
\]

Now let $H\notin\mathcal P_R$.  If $H$ does not contain $R$, then it shares no ridge with any facet in $\mathcal P_R$, and hence $(Mx)_H=0$.  If $R\subset H$, and $\varepsilon_H$ denotes the incidence sign of $R$ in $\partial_dH$, then
\[
(Mx)_H
=\varepsilon_H\sum_{F\in\mathcal P_R}\varepsilon_Fx_F
=0.
\]
Since $x_H=0$, both cases give $(Mx)_H=dx_H$.  Thus
\[
U_R\subseteq\ker(M-dI).
\]

For distinct quasi-pendant ridges $R$ and $S$, the sets $\mathcal P_R$ and $\mathcal P_S$ are disjoint.  Hence the spaces $U_R$ have disjoint supports and their sum is direct.  Consequently,
\[
\begin{aligned}
m_{\Upsilon}(d)
&\ge\sum_R\dim U_R\\
&=\sum_R(t_R-1)\\
&=p_d(\Upsilon)-q_d(\Upsilon).
\end{aligned}
\]
\end{proof}

The bound is sharp in every dimension.  Indeed, for a ridge-star with $m$ facets,
\[
p_d(\Upsilon)=m,\qquad q_d(\Upsilon)=1,
\]
and the calculation
\[
B^{\mathsf T}B=dI_m+J_m
\]
gives $m_{\Upsilon}(d)=m-1=p_d(\Upsilon)-q_d(\Upsilon)$.

When $d=1$, facets are edges, ridges are vertices, pendant facets correspond to pendant vertices, and quasi-pendant ridges correspond to quasi-pendant vertices.  Therefore Theorem~\ref{thm:simplicial-faria} becomes
\[
m_G(1)\ge p(G)-q(G),
\]
which is precisely Faria's inequality \cite{Faria}.

\section{Concluding remarks}

We have obtained upper bounds for the multiplicities of nonzero upper-Laplacian eigenvalues of pure ridge-connected simplicial complexes in terms of the numbers of facets, ridges, pendant facets and quasi-pendant ridges.  The equality cases show that simplicial cycles and ridge-stars play the same extremal roles as cycles and stars in the graph setting.  We have also established a simplicial analogue of Faria's inequality for the exceptional eigenvalue $d$.

It would be interesting to investigate whether similar multiplicity bounds hold for normalized or weighted simplicial Laplacians, as well as for the down and full Hodge Laplacians.  Another natural problem is to determine whether stronger bounds can be obtained for other restricted classes of simplicial complexes.

\section*{Declarations}

\noindent
\textbf{Funding.} The author received no financial support for the research, authorship, or publication of this article.

\noindent
\textbf{Conflict of interest.} The author declares no conflict of interest.

\noindent
\textbf{Data availability.} No datasets were generated or analyzed in this study.

\end{document}